\documentclass[a4paper,9pt]{amsart}
\usepackage[T1]{fontenc}
\usepackage[utf8]{inputenc}
\usepackage{amssymb}
\usepackage{graphicx}
\usepackage{geometry}
\newtheorem{Theorem}{Theorem}
\newtheorem{Lemma}[Theorem]{Lemma}
\begin{document}

\title[Consensus in the Hegselmann-Krause model]{How long does it take to consensus in the Hegselmann-Krause model?}

\author{Sascha Kurz}
\address{Sascha Kurz, Mathematisches Institut, Universit\"at Bayreuth, 95440 Bayreuth, Germany. E-mail: sascha.kurz@uni-bayreuth.de, 
     Phone: +49\,921\,55\,7353, Fax: +49\,921\,55\,7353, Homepage: http://www.wm.uni-bayreuth.de/index.php?id=sk}
\begin{abstract}
  Hegselmann and Krause introduced a discrete-time model of opinion dynamics with agents having limit confidence. It is well known
  that the dynamics reaches a stable state in a polynomial number of time steps. However, the gap between the known lower and 
  upper bounds for the worst case is still immense. In this paper exact values for the maximum time, needed to reach consensus or 
  to discover that consensus is impossible, are determined 
  using an integer linear programming approach.
\end{abstract}
\maketitle                   

\section{Introduction and problem formulation}
\label{sec_introduction}
A Hegselmann-Krause system on the real line, HK system for short, is defined as follows. There are $n$ agents with real-valued initial 
opinions $x_1(0)\le x_2(0)\le\dots\le x_n(0)$. For all $t\in\mathbb{N}$ the opinion of agent $i$ at time $t+1$ is given by 
$x_i(t+1)=\left(\sum_{j\in\mathcal{N}_i(t)} x_j(t)\right)/\left|\mathcal{N}_i(t)\right|$, where 
$\mathcal{N}_i(t)=\left\{j\,:\,\left|x_i(t)-x_j(t)\right|\le 1\right\}$ (cf.~\cite{hegselmann2002opinion}). Obviously, the ordering of the
agents, i.e., $x_1(t)\le x_2(t)\le\dots\le x_n(t)$, is preserved over time. We say that a HK system has
converged at time $T$ if $x_i(t+1)=x_i(t)$ holds for all agents and all time steps $t\ge T$. For the convergence time,
i.e., the smallest time $T$ such that the system has converged at $T$, an upper bound of $O(n^3)$ was proven 
\cite{bhattacharyya2013convergence,mohajer2012convergence}. It is conjectured that the worst-case scenario is given by the initial
positions $x_i(0)=i$ -- \emph{equidistant configuration} for later reference -- and needs roughly $\frac{4n}{5}$ time steps 
\cite{wedin2014hegselmann}.

\begin{figure}[htp]
  \centering
  \includegraphics[width=0.2\linewidth]{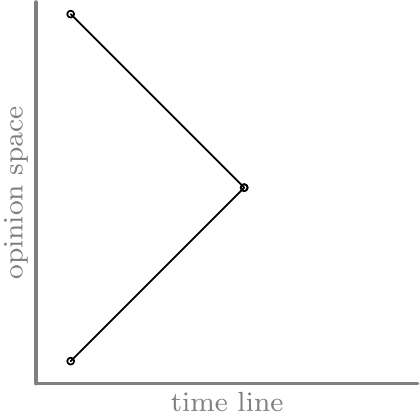}%
  \includegraphics[width=0.2\linewidth]{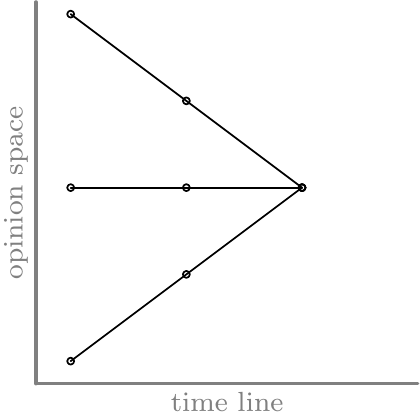}%
  \includegraphics[width=0.2\linewidth]{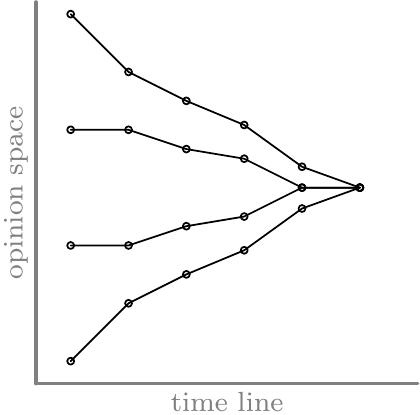}%
  \includegraphics[width=0.2\linewidth]{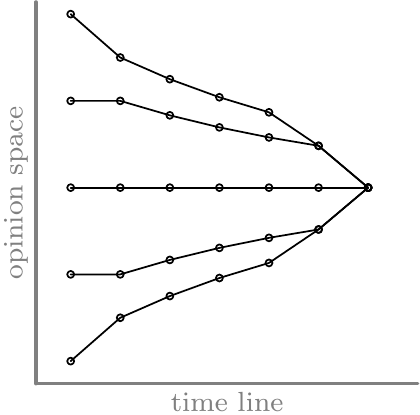}%
  \includegraphics[width=0.2\linewidth]{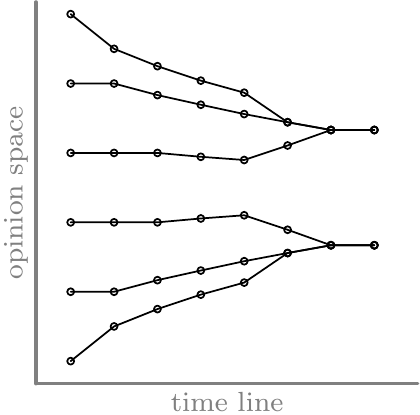}%
  \caption{Trajectories for the equidistant configuration with $n=2,\dots,6$ agents}
  \label{fig:trajectories_equidistant}
\end{figure}

\noindent
In an equidistant configuration with at most $5$ agents the final state consists of a single opinion, see Figure~\ref{fig:trajectories_equidistant}. 
There is also the possibility that different clusters arise, as e.g.\ for $n=6$. We define the weight of $x\in\mathbb{R}$ at time $t$ as 
$w_t(x)=\left\{i\,:\,x_i(t)=x\right\}$ and the the weight of agent $i$ at time $t$ as $w_t(x_i(t))$. With this, we say that a HK system has reached 
consensus at time $t$ if the weight of agent $1$ equals $n$. Whenever we have $x_{i+1}(t)-x_i(t)>1$, we have $x_{i+1}(t')-x_i(t')>1$ for all 
$t'\ge t$, i.e., no consensus is possible. By $f(\mathcal{C})$ we denote the earliest time $t$ such that a given HK system $\mathcal{C}$ has 
reached consensus or we have $x_{i+1}(t)-x_i(t)>1$ for at least one index $i$. With this let $f(n)$ denote the maximum of $f(\mathcal{C})$ over 
all configurations consisting of $n$ agents. Clearly $f(\mathcal{C})$ is upper bounded by the convergence time, so that $f(n)\in O(n^3)$.


\section{An analytical lower bound}
\label{sec_analytical}
Let $k\ge 4$ be an integer, $j_1=1$, $j_2=k+1$, $j_3=k+2$, $j_4=k+3$, and $n=2k+2$. We consider the HK system $\mathcal{C}$ with initial opinions 
$x_i(0)=-\frac{1}{k}$ for all $1\le i\le k$, $x_{j_2}(0)=0$, $x_{j_3}(0)=1$, and $x_i(0)=1+\frac{1}{k}$ for all $j_4\le i\le n$, i.e., 
agents $j_1$ and $j_4$ have weight $k$ and are rather close to agents $j_2$ and $j_3$, respectively. From the following lemma 
we conclude $f(\mathcal{C})\in\Omega(n)$.
\begin{Lemma}
  Let $a_t=t+1$, $b_t=\frac{(t+1)(t+2)}{2}$, and $c_t=t$. For all $0\le t\le \frac{k}{3}$ we have 
  $$
    -\frac{b_t}{k^3}\le x_{j_1}(t)+\frac{1}{k}-\frac{a_t}{k^2}\le 0,\quad
    0\le x_{j_2}(t)\le \frac{c_t}{k^2},\quad 1-\frac{c_t}{k^2}\le x_{j_3}(t)\le 1,\quad
    0\le x_{j_4}(t)-1+\frac{1}{k}-\frac{a_t}{k^2}\le \frac{b_t}{k^3}.
  $$
\end{Lemma}
\begin{proof}
  If true, we have $x_{j_2}(t)-x_{j_1}(t)\le \frac{c_t}{k^2}+\frac{1}{k}\le 1$, $x_{j_3}(t)-x_{j_2}(t)\le 1$, and 
  $x_{j_3}(t)-x_{j_1}(t)\ge 1-\frac{c_t}{k^2}+\frac{1}{k}-\frac{a_t}{k^2}>1$ for all $0\le t\le \frac{k}{3}$, i.e., 
  $\mathcal{N}_{j_1}=\{1,\dots,j_1,j_2\}$, $\mathcal{N}_{j_2}=\{1,\dots,j_1,j_2,j_3\}$, $\mathcal{N}_{j_3}=\{j_2,j_3,j_4,\dots,n\}$, and
  $\mathcal{N}_{j_4}=\{j_3,j_4,\dots,n\}$ due to symmetry. Next we proof the assertion by induction on $t\le\frac{k}{3}$, where we use 
  $\frac{1}{k}-\frac{1}{k^2}\le \frac{1}{k+1}\le \frac{1}{k}-\frac{1}{k^2}+\frac{1}{k^3}$ and $x_2(t)+x_3(t)=1$ (symmetry again). 
  We can check that for $t=0$ the asserted inequalities are satisfied. For $t\ge 0$ we have 
  \begin{eqnarray*}
    x_{j_1}(t+1) &=&\frac{k\cdot x_{j_1}(t)+x_{j_2}(t)}{k+1}
    \le \frac{k\cdot\left(-\frac{1}{k}+\frac{a_t}{k^2}\right)+\frac{c_t}{k^2}}{k+1}
    \le \left(\frac{1}{k}-\frac{1}{k^2}\right)\cdot\left(-1+\frac{a_t}{k}+\frac{c_t}{k^2}\right)
    \le -\frac{1}{k}+\frac{a_{t}+1}{k^2},\\
    x_{j_1}(t+1) &\ge&
    \frac{k\cdot\left(-\frac{1}{k}+\frac{a_t}{k^2}-\frac{b_t}{k^3}\right)+0}{k+1}
    \ge\left(\frac{1}{k}-\frac{1}{k^2}+\frac{1}{k^3}\right)\cdot\left(-1+\frac{a_t}{k}-\frac{b_t}{k^2}\right)
    \ge -\frac{1}{k}+\frac{a_t+1}{k^2}-\frac{a_t+b_t+1}{k^3},\\
    x_{j_2}(t+1) &=& \frac{k\cdot x_{j_1}(t)+x_{j_2}(t)+x_{j_3}(t)}{k+2}
    \ge \frac{k\cdot\left(-\frac{1}{k}+\frac{a_t}{k^2}-\frac{b_t}{k^3}\right)+1}{k+2}
    = \frac{a_t-\frac{b_t}{k}}{k(k+2)}\ge 0,\text{ and}\\
    x_{j_2}(t+1) &\le& \frac{k\cdot\left(-\frac{1}{k}+\frac{a_t}{k^2}\right)+1}{k+2}= \frac{a_t}{k(k+2)}\le \frac{a_t}{k^2},
  \end{eqnarray*}
  where we always assume $t\le\frac{k}{3}$ and $k\ge 4$. The remaining inequalities follow from the symmetry of the configuration.
\end{proof} 

\section{An integer linear programming model}
\label{sec_ilp_model}

By considering $\mathcal{N}_i(t)\backslash\{i\}$ as the neighbors of agent~$i$, we obtain the \emph{influence graph} $\mathcal{G}(t)$, 
which is clearly a unit interval graph. Let $\mathcal{I}_n^c$ denote the set of connected unit interval graphs with vertex set 
$V=\{1,\dots,n\}$, which admit a representation satisfying $x_i\le x_j$ for all $i\le j$, and $K_n$ denote the complete 
graph on $n$ vertices. Given $n,T\in\mathbb{N}$ and $\varepsilon\in\mathbb{R}$ consider the BLP:\\[-5mm]
\begin{eqnarray*}
  \min \sum_{I=(V,E)\in\mathcal{I}_n^c} |E|\cdot u_I^T &\!\!\!& \text{subject to}\\
  x_j^t-x_i^t\le 1+\varepsilon + (1-u_I^t)\cdot n &\!\!\!& \forall 0\le t\le T,I=(V,E)\in \mathcal{I}_n^c, 1\le i<j\le n \text{ with } \{i,j\}\in E\\
  x_j^t-x_i^t\ge 1-\varepsilon -(1-u_I^t)\cdot (1-\varepsilon) &\!\!\!& \forall 0\le t\le T,I=(V,E)\in \mathcal{I}_n^c, 1\le i<j\le n 
  \text{ with } \{i,j\}\notin E\\
  \sum_{I\in\mathcal{I}_n^c}\!\! u_I^t=1\quad \forall 0\le t\le T\quad u_{K_n}^t=0 \quad\forall 0\le t<T
  &\!\!\!&
  x_i^t=\!\!\!\!\!\!\sum_{I=(V,E)\in\mathcal{I}_n^c}\!\!\!\!\!\!\frac{x_{i}^{t-1}+\sum_{\{i,j\}\in E}z_{j,I}^{t-1}}{1+\sum_{\{i,j\}\in E}1} 
  \quad\forall 1\le t\le T,i\in V\\
  z_{i,I}^t\le n\cdot u_{I}^t\quad\forall 0\le t<T,\,I\in\mathcal{I}_n^c\,i\in V&&
  z_{i,I}^t \ge x_i^t-n\cdot(1-u_I^t)\quad \forall 0\le t<T, I\in \mathcal{I}_n^c,\,i\in V\\
  z_{i,I}^t\le x_i^t\quad\forall 0\le t<T,\,I\in\mathcal{I}_n^c,\,i\in V&&
  z_{i,I}^t\in[0,n]\quad\forall 0\le t< T,\, I\in\mathcal{I}_n^c,\,i\in V\\
  u_{I}^t\in\{0,1\}\quad\forall 0\le t\le T,I\in\mathcal{I}_n^c &\!\!\!& x_i^t\in[0,n]\quad\forall 0\le t\le T,\,i\in V
\end{eqnarray*}
If it has a solution for $\varepsilon<0$ and $T\ge 1$, then we have $f(n)\ge T+1$. If it has no solution for $\varepsilon\ge 0$, then we 
have $f(n)\le T$. By choosing $\varepsilon$ apart from zero we can partially prevent from wrong results, caused by 
numerical inaccuracies (cf.~\cite{optimal_control}).

\section{Conclusion and open problems}
\label{sec_conclusion}
We have defined the maximum time $f(n)$ to a consensus in a HK system of $n$ agents. Our analytical example for the proof of $f(n)\in\Omega(n)$ 
also shows that it might be hard to decrease the $O(n^3)$ bound for the maximum convergence time, since the known approaches rely on the fact 
that the leftmost or the rightmost agent moves by at least some $\delta$ in a finite amount of time steps. In our example we have 
$\delta\in O(\frac{1}{n^2})$ for $\Omega(n)$ rounds while the distance between two connected agents can be in $\Omega(n)$. Using an 
integer linear programming model we have determined $f(n)=0,1,2,5,7,9,\ge\!\!12$ for $1\le n\le 7$. It turned out that for at least 
$n\in\{5,6,\dots,14\}$ the equidistant configuration does not lead to the maximum convergence time -- disproving a conjecture of 
\cite{wedin2014hegselmann} \footnote{Indeed, for $n\ge2$ the convergence time is given by $1+5\left\lfloor\frac{n+2}{6}\right\rfloor +
\frac{1}{3}\left(\sqrt{3}\sin\left(\frac{2\pi(n-1)}{3}\right)-\cos\left(\frac{\pi(n-1)}{3}\right)-(-1)^n\right)$, which tends to $\frac{5n}{6}$.}.
Nevertheless we conjecture the maximum convergence time, and so also $f(n)$, to be linearly bounded. Since 
$\left|\mathcal{I}_n^c\right|={{2n-2}\choose{n-1}}/n\sim \frac{4^{n-1}}{\sqrt{\pi n}\cdot n}$, a column-generation approach might by 
very beneficial.

The problem becomes even more challenging if we allow the presence of external controls cf.~\cite{chazelle2011total,optimal_control,kurz2011hegselmann}.



\bibliographystyle{amsplain}

\end{document}